\documentclass{amsart}
\usepackage{amsmath}
\usepackage{amsthm}
\usepackage{amsfonts}
\usepackage{amssymb}
\usepackage{amsbsy}
\usepackage{amscd}
\usepackage{eucal}
\usepackage{amsaddr}

\newtheorem{theorem}{Theorem}

\numberwithin{equation}{section}

\newcommand{\beq}{\begin{equation}}
\newcommand{\eeq}{\end{equation}}

\newcommand{\rmd}{\mathrm{d}}
\newcommand{\rmi}{\mathrm{i}}

\newcommand{\R}{\mathbb{R}}

\newcommand{\hf}{\widehat{f}}

\newcommand{\sds}{\strut\displaystyle}
\newcommand{\ud}{\frac{1}{2}}

\newcommand{\skd}{\vspace*{0.2cm}}
\newcommand{\skt}{\vspace*{0.3cm}}

\begin{document}

\title[Legendre coefficients]{A new and efficient method for the computation of Legendre coefficients}

\author[E. De Micheli]{Enrico De Micheli$^*$}
\address{\sl Consiglio Nazionale delle Ricerche \\ Via De Marini, 6 - 16149 Genova, Italy \\
E-mail: enrico.demicheli@cnr.it}
\thanks{{$^*$}Corresponding author.}
\author[G. A. Viano]{Giovanni Alberto Viano}
\address{\sl Dipartimento di Fisica -- Universit\`a di Genova,\\
Istituto Nazionale di Fisica Nucleare -- Sezione di Genova, \\
Via Dodecaneso, 33 - 16146 Genova, Italy \\
E--mail: viano@ge.infn.it}


\begin{abstract}
An efficient procedure for the computation of the coefficients of Legendre expansions
is here presented. We prove that the Legendre coefficients associated with a function $f(x)$
can be represented as the Fourier coefficients of an Abel--type transform of $f(x)$. 
The computation of $N$ Legendre coefficients can then be performed in 
$\mathcal{O}(N\log N)$ operations with a single Fast Fourier Transform of the 
Abel--type transform of $f(x)$.
\end{abstract}

\maketitle

\section{Introduction}
\label{se:introduction}

The efficient computation of the coefficients of Legendre expansions
is a very important problem in numerical analysis and applied mathematics
with a wide range of applications including, just to mention a few,
approximation theory, solution of partial differential equations and quadratures. 
Recently its relevance emerged also in connection with the computation of
spectra of highly oscillatory Fredholm integral operators, which play an important role
in laser engineering \cite{Brunner}. 

The difficulty of the problem lies essentially
in the fact that these coefficients are represented by integrals whose integrands
oscillate rapidly for large values of the index of the polynomials. Using standard quadrature
procedures for the calculation of $N$ Legendre coefficients leads only to \emph{slow} $\mathcal{O}(N^2)$ algorithms
(see, e.g., Ref. \cite{Delic}).
More efficiently, in Ref. \cite{Alpert} (see also \cite{Piessens,Potts}) 
the Legendre coefficients are obtained by a suitable transformation
of the corresponding Chebyshev coefficients, which yields \emph{faster} $\mathcal{O}(N(\log N)^2)$ algorithms.
Recently this problem has been also clearly discussed in a paper by A. Iserles \cite{Iserles1}, in which
an algorithm for the computation of the Legendre coefficients, which is certainly fast and brilliant, is presented. 

In this paper we present an alternative procedure.
The basic idea of our method consists in exploiting the Dirichlet--Murphy
integral representation of the Legendre polynomials. Next, we prove that
the coefficients of the Legendre expansion of a function $f(x)$
are connected with a subset of the Fourier coefficients (the ones with nonnegative index) 
of an Abel--type transform of $f(x)$.

The numerical implementation of the algorithm follows straightforwardly and is very efficient.
The aforementioned Fourier coefficients (which represent the searched Legendre coefficients) can be 
computed in $\mathcal{O}(N\,\log N)$ operations by a single Fast Fourier Transform 
after the evaluation of the Abel--type integral by means of standard quadrature techniques.

\section{Connection of Legendre expansions to Fourier series}
\label{se:interpretation}

The standard form of the Legendre expansion reads:
\beq
f(x) = \sum_{n=0}^\infty c_n \, P_n(x) \qquad x\in[-1,1],
\label{1}
\eeq
where $P_n(x)$ are the Legendre polynomials,
which can be defined by the generating function \cite{Bateman}:
\beq
\sum_{n=0}^\infty P_n(x)\,t^n = \left(1-2xt+t^2\right)^{-\ud},
\label{5bis}
\eeq
and the coefficients $\{c_n\}_{n=0}^\infty$ are given by:
\beq
c_n = \left(n+\ud\right)\int_{-1}^1 f(x)\,P_n(x)\,\rmd x
\qquad (n\geqslant 0).
\label{2}
\eeq 
The conditions to be satisfied by $f(x)$ to guarantee the uniform convergence of the series in 
\eqref{1} are discussed in \cite{Hobson}. However, for our purpose of computing
the Legendre coefficients $c_n$ it is sufficient to assume that $f(x)$ be summable
in the interval $[-1,1]$.

We can now state the following theorem.

\begin{theorem}
\label{the:1}
The coefficients $\{a_n\}_{n=0}^\infty$, defined as $a_n\doteq\frac{c_n}{(2n+1)}$, 
coincide with the Fourier coefficients (with $n\geqslant 0$) of an
Abel--type transform of $f(x)$, that is:
\beq
a_n \doteq \frac{c_n}{(2n+1)} = \int_{-\pi}^\pi\hf(y) \, e^{\rmi ny}\,\rmd y \qquad (n\geqslant 0),
\label{6}
\eeq
where the $2\pi$-periodic function $\hf(y)$ is defined by
\beq
\hf(y) = \frac{1}{2\pi\rmi}\,\varepsilon(y)\,e^{\rmi\frac{y}{2}}\!
\int_{\cos y}^1 \,\frac{f(x)}{[2(x-\cos y)]^\ud}\,\rmd x \qquad (y\in\R),
\label{7}
\eeq
$\varepsilon(y)$ being the sign function.
\end{theorem}

\begin{proof} 
Plugging the Dirichlet--Murphy integral representation of the Legendre 
polynomials \cite[Ch. III, \S 5.4]{Vilenkin}:
\beq
P_n(\cos x)=-\frac{\rmi}{\pi}\int_x^{(2\pi-x)}
\frac{e^{\,\rmi(n+\ud)y}}{\left[2(\cos x - \cos y)\right]^\ud}\,\rmd y,
\label{8}
\eeq
into equality \eqref{2} (after the change of variable $x \to \cos x$),
we have:
\beq
2\pi\rmi \, a_n=
\int_0^\pi \rmd x\,f(\cos x)\sin x\,
\int_x^{(2\pi-x)}\frac{e^{\,\rmi(n+\ud)y}}
{\left[2(\cos x - \cos y)\right]^{\ud}}\,\rmd y.
\label{9}
\eeq
Interchanging the order of integration in \eqref{9} we have:
\beq
\begin{split}
& 2\pi\rmi\, a_n=\int_0^\pi \rmd y \ e^{\,\rmi(n+\ud)y}
\int_0^y f(\cos x)\,\frac{\sin x}{\left[2(\cos x - \cos y)\right]^{\ud}}\,\rmd x \\
&\quad +\int_\pi^{2\pi} \rmd y \ e^{\,\rmi(n+\ud)y} 
\int_0^{(2\pi-y)} f(\cos x) \ \frac{\sin x}{\left[2(\cos x - \cos y)\right]^\ud}\, \rmd x.
\end{split}
\label{10}
\eeq
Next, if we make the change of variables: $y\to y-2\pi$ and $x \to -x$,
the second integral on the r.h.s. of \eqref{10} becomes: 
\beq
e^{\rmi\pi}\int_{-\pi}^0 \rmd y \ e^{\,\rmi(n+\ud)y}
\int_0^y f(\cos x) \ 
\frac{\sin x}{\left[2(\cos x - \cos y)\right]^\ud}\, \rmd x.
\label{11}
\eeq
Finally, we obtain:
\beq
\begin{split}
& 2\pi\rmi\, a_n=
\int_0^\pi \rmd y \ e^{\,\rmi(n+\ud)y}
\int_0^y f(\cos x) \ \frac{\sin x}{\left[2(\cos x - \cos y)\right]^\ud}\, \rmd x \\
& \quad +e^{\rmi\pi}\int_{-\pi}^0 \rmd y \ e^{\,\rmi(n+\ud)y}
\int_0^y f(\cos x) \ \frac{\sin x}{\left[2(\cos x - \cos y)\right]^\ud}\, \rmd x,
\end{split}
\label{12}
\eeq
which, after the change of variable $\cos x \to x$ into the integrals on the r.h.s., yields:
\beq
a_n=\int_{-\pi}^\pi \hf(y) \,e^{\rmi ny}\,\rmd y \qquad (n\geqslant 0),
\label{13}
\eeq
with $\hf(y)$ given by \eqref{7}.
\end{proof}

\skd

It is easy to check from \eqref{7} that $\hf(y)$ satisfies the following symmetry relation:
\beq
\hf(y)=-e^{\rmi y}\, \hf(-y).
\label{14}
\eeq
This latter, along with formulae \eqref{6} and \eqref{7}, allows us to write the Legendre
coefficients $c_n$ in the following form:
\beq
c_n = \frac{2}{\pi}\left(n+\ud\right)\int_0^\pi \phi(y)\sin\left[\left(n+\ud\right)y\right]\,\rmd y,
\label{15}
\eeq
where
\beq
\phi(y) = \int_{\cos y}^1 \frac{f(x)}{[2(x-\cos y)]^\ud}\,\rmd x.
\label{16}
\eeq

\skt

The numerical implementation of the algorithm first requires the computation of the
Abel--type integral $\hf(y)$ defined in \eqref{7} (or, equivalently, of the function $\phi(y)$ in \eqref{16}). 
The integrand presents a weak algebraic singularity at the end point of the domain of integration, which
can be effectively handled by means of a proper nonlinear change of variable. This technique, along
with the use of a standard quadrature formula (e.g., the Gauss-Legendre one), 
allows obtaining high accuracy with a small number of nodes \cite{Monegato}.

Finally, formula \eqref{6} makes it possible to take full advantage of the computational efficiency of
the Fast Fourier Transform both in terms of speed of computation and of accuracy \cite{Calvetti,Kaneko}. 
The calculation of the first $N$ coefficients of the expansion can consequently be accomplished 
in $\mathcal{O}(N\log N)$ operations.

The algorithm described has been implemented in double precision arithmetics using
the open source GNU Scientific Library (GSL) \cite{Galassi}, and
its performance has been tested on a variety of functions.

First, feasibility and accuracy of the algorithm have been verified by direct comparison 
of the obtained numerical results with the true Legendre coefficients for the function 
$f(x)=|x|^{3/2}$, whose Legendre coefficients are known to be \cite[p. 78]{Canuto}:
\beq
c_n =
\begin{cases}
0 &      \quad\textrm{if $n$ odd}, \\[+5pt]
(\alpha+1)^{-1} & \quad\textrm{if $n=0$}, \\[+5pt]
\frac{\sds(2n+1)\,\alpha\,(\alpha-2)\cdots(\alpha-n+2)}{\sds(\alpha+1)(\alpha+3)\cdots(\alpha+n+1)} 
& \quad\textrm{otherwise,}
\end{cases}
\eeq
where $\alpha=3/2$. Values of the computed Legendre coefficients along with the absolute error are given
in Table \ref{tab:1}.

\begin{table}[tb]
\caption{\label{tab:1} \it True and computed Legendre coefficients $c_n$ for the function $f(x)=|x|^{3/2}$.}
\centering
\begin{tabular}{|c|r|r|c|}
\hline
\rule{0pt}{4ex} $\boldsymbol{n\,}$ & \bf True $\boldsymbol{c_n}$ \hspace{1.3cm} 
& \bf Computed $\boldsymbol{c_n}$ \hspace{0.6cm} & \bf Error \\[+5pt]
\hline\rule{0pt}{3ex}
0  & 0.40000000000000000000 & 0.40000000000268187694 & 2.68$_{-12}$ \\[+5pt]
2  & 0.66666666666666666666 & 0.66666666671084839901 & 4.42$_{-11}$ \\[+5pt]
4  &-0.09230769230769230769 &-0.09230769246022409169 & 1.53$_{-10}$ \\[+5pt] 
6  & 0.03921568627450980392 & 0.03921568630769028951 & 3.32$_{-11}$ \\[+5pt] 
8  &-0.02197802197802197802 &-0.02197802226184656509 & 2.84$_{-10}$ \\[+5pt]
10 & 0.01411764705882352941 & 0.01411764762633771833 & 5.68$_{-10}$ \\[+5pt]
12 &-0.00985221674876847290 &-0.00985221698441776129 & 2.36$_{-10}$ \\[+5pt]
14 & 0.00727272727272727272 & 0.00727272772068757959 & 4.48$_{-10}$ \\[+5pt]
16 &-0.00559179869524697110 &-0.00559179938291245667 & 6.88$_{-10}$ \\[+5pt]
18 & 0.00443458980044345898 & 0.00443458980116155730 & 7.18$_{-13}$ \\[+5pt]
20 &-0.00360360360360360360 &-0.00360360382001508440 & 2.16$_{-10}$ \\[+5pt]
22 & 0.00298656047784967645 & 0.00298656131882347585 & 8.41$_{-10}$ \\[+5pt]
24 &-0.00251572327044025157 &-0.00251572347582495201 & 2.05$_{-10}$ \\[+5pt]
26 & 0.00214822771213748657 & 0.00214822836757538881 & 6.55$_{-10}$ \\[+5pt]
28 &-0.00185586142901330034 &-0.00185586305079012375 & 1.62$_{-09}$ \\[+5pt]
30 & 0.00161943319838056680 & 0.00161943394840411781 & 7.50$_{-10}$ \\[+5pt]
\hline
\end{tabular} 
\end{table}

The increment of performances, with respect to the computation of the Legendre coefficients $c_n$ 
by ordinary quadrature, has been verified in terms of speed of computation at (nearly) equality of precision.
All accuracies have been determined by comparing the results of the algorithm with the reference values 
of $c_n$, computed with $20$ significant figures by standard quadrature with Mathematica \cite{Mathematica}.
For these tests we used various functions (many of them have been already used in previous works), 
including polynomials, exponential/hyperbolic functions, 
rational functions (e.g., $f(x)=\frac{1+x}{\gamma^2+x^2}$ with $\gamma=$ constant).
All the results have confirmed the enormous increase of computational speed 
(the expected improvement ratio being proportional to $N/\log N$).
Such an increase of performances will become even more crucial for the efficient evaluation of multivariate Legendre
transform \cite{Brunner} and expansions in Gegenbauer (alias Ultraspherical Legendre) polynomials, which will
be the subject of a forthcoming paper.

\end{document}